\documentclass[12pt]{article}
\usepackage{amsfonts}
\usepackage{amsthm}
\usepackage{amsmath}
\usepackage{indentfirst}
\usepackage[top=1em,bottom=1em,left=6em,right=6em,includehead,includefoot]{geometry}
\begin{document}

\newtheorem{theorem}{Theorem}
\newtheorem{lemma}{Lemma}
\newtheorem{definition}{Definition}
\newtheorem{remark}{Remark}
\newtheorem{corollery}{Corollery}

\title{\bf Localization formulas about two Killing vector fields}
\author{Xu Chen \footnote{{\it Email:} xiaorenwu08@163.com.  ChongQing, China. }}
\date{}
\maketitle

\begin{abstract}
In this article, we will discuss the smooth
$(X_{M}+\sqrt{-1}Y_{M})$-invariant forms on $M$ and to establish a
localization formulas. As an application, we get a localization formulas for characteristic numbers.
\end{abstract}

The localization theorem for equivariant differential forms was
obtained by Berline and Vergne(see [2]). They discuss on the zero
points of a Killing vector field, the localization formula expresses the integral of an equivariantly closed differential form as an integral over the set of zeros of the Killing vector field. Now, We will discuss the equivariant cohomology about two Killing vector fields and to establish a localization formulas.

Let $M$ be a smooth closed oriented manifold. Let $G$ be
a compact Lie group acting smoothly on $M$, and let $\mathfrak{g}$
be its Lie algebra. Let $g^{TM}$ be a $G$-invariant metric on $TM$.
If $X,Y\in\mathfrak{g}$, let $X_{M} ,Y_{M}$ be the corresponding
smooth vector field on $M$. If $X,Y\in\mathfrak{g}$, then $X_{M}
,Y_{M}$ are Killing vector field. Here we will introduce the
equlvariant cohomology about two Killing vector fields.

\section{Equlvariant cohomology by two Killing vector fields}
   First, let us review the definition of equlvariant cohomology by a Killing vector field. Let
$\Omega^{*}(M)$ be the space of smooth differetial forms on $M$, the
de Rham complex is $(\Omega^{*}(M),d)$. Let $L_{X_{M}}$ be the Lie
derivative of $X_{M}$ on $\Omega^{*}(M)$, $i_{X_{M}}$ be the
interior multiplication induced by the contraction of $X_{M}$.\par
Set
$$d_{X}=d+i_{X_{M}},$$ then $d_{X}^{2}=L_{X_{M}}$ by the following
Cartan formula $$L_{X_{M}}=[d,i_{X_{M}}].$$\par

Let
$$\Omega_{X}^{*}(M)=\{\omega\in\Omega^{*}(M):L_{X_{M}}\omega=0\}$$
be the space of smooth $X_{M}$-invariant forms on $M$. Then
$d_{X}^{2}\omega=0$, when $\omega\in\Omega_{X}^{*}(M)$. It is a
complex $(\Omega_{X}^{*}(M),d_{X})$. The corresponding cohomology
group
$$H^{*}_{X}(M)=\frac{\rm{Ker}d_{X}|_{\Omega_{X}^{*}(M)}}{\rm{Im}d_{X}|_{\Omega_{X}^{*}(M)}}$$
is called the equivariant cohomology associated with $X$. If a form
$\omega$ has $d_{X}\omega=0$, then $\omega$ called $d_{X}$-closed
form.\par

Then we will to definite a new complex by two Killing vector field.
If $X,Y\in\mathfrak{g}$, let $X_{M} ,Y_{M}$ be the corresponding
smooth vector field on $M$.\par We know
$$L_{X_{M}}+\sqrt{-1}L_{Y_{M}}$$
be the operator on
$\Omega^{*}(M)\otimes_{\mathbb{R}}\mathbb{C}$.\par

Set
$$i_{X_{M}+\sqrt{-1}Y_{M}}\doteq i_{X_{M}}+\sqrt{-1}i_{Y_{M}}$$
be the interior multiplication induced by the contraction of
$X_{M}+\sqrt{-1}Y_{M}$. It is also a operator on
$\Omega^{*}(M)\otimes_{\mathbb{R}}\mathbb{C}$.\par

Set$$d_{X+\sqrt{-1}Y}=d+i_{X_{M}+\sqrt{-1}Y_{M}}.$$

\begin{lemma}
If $X,Y\in\mathfrak{g}$, let $X_{M} ,Y_{M}$ be the corresponding
smooth vector field on $M$; then
$$d_{X+\sqrt{-1}Y}^{2}=L_{X_{M}}+\sqrt{-1}L_{Y_{M}}$$
\end{lemma}
\begin{proof}
\begin{align*}
(d+i_{X_{M}+\sqrt{-1}Y_{M}})^{2}
&=(d+i_{X_{M}}+\sqrt{-1}i_{Y_{M}})(d+i_{X_{M}}+\sqrt{-1}i_{Y_{M}})\\
&=d^{2}+di_{X_{M}}+i_{X_{M}}d+\sqrt{-1}di_{Y_{M}}+\sqrt{-1}i_{Y_{M}}d+(i_{X_{M}}+\sqrt{-1}i_{Y_{M}})^{2}\\
&=L_{X_{M}}+\sqrt{-1}L_{Y_{M}}
\end{align*}
\end{proof}
Let$$\Omega_{X_{M}+\sqrt{-1}Y_{M}}^{*}(M)=\{\omega\in\Omega^{*}(M)\otimes_{\mathbb{R}}\mathbb{C}:(L_{X_{M}}+\sqrt{-1}L_{Y_{M}})\omega=0\}$$
be the space of smooth $(X_{M}+\sqrt{-1}Y_{M})$-invariant forms on
$M$. Then we get a complex $(\Omega_{X_{M}+\sqrt{-1}Y_{M}}^{*}(M),
d_{X+\sqrt{-1}Y})$. We call a form $\omega$ is
$d_{X+\sqrt{-1}Y}$-closed if $d_{X+\sqrt{-1}Y}\omega=0$ (this is
first discussed by Bimsut, see [3]).The corresponding cohomology
group
$$H^{*}_{X+\sqrt{-1}Y}(M)=\frac{{\rm{Ker}}d_{X+\sqrt{-1}Y}|_{\Omega_{X+\sqrt{-1}Y}^{*}(M)}}{{\rm{Im}}d_{X+\sqrt{-1}Y}|_{\Omega_{X+\sqrt{-1}Y}^{*}(M)}}$$
is called the equivariant cohomology associated with $K$.

\section{The set of zero points}
\begin{lemma}
If $X,Y\in\mathfrak{g}$, let $X_{M} ,Y_{M}$ be the corresponding
smooth vector field on $M$, $X^{'}, Y^{'}$ be the 1-form on $M$
which is dual to $X_{M} ,Y_{M}$ by the metric $g^{TM}$, then
$$L_{X_{M}}Y^{'}+L_{Y_{M}}X^{'}=0.$$
If $[X_{M},Y_{M}]=0$, then $$L_{X_{M}}Y^{'}=L_{Y_{M}}X^{'}=0.$$
\end{lemma}
\begin{proof}
Because
$$(L_{X_{M}}\omega)(Z)=X_{M}(\omega(Z))-\omega([X_{M},Z])$$
here $Z\in\Gamma(TM)$, So we get
$$(L_{X_{M}}Y^{'})(Z)=X_{M}<Y_{M},Z>-<[X_{M},Z],Y_{M}>$$
$$(L_{Y_{M}}X^{'})(Z)=Y_{M}<X_{M},Z>-<[Y_{M},Z],X_{M}>.$$
Because $X_{M},Y_{M}$ are Killing vector fields, so (see [6])
\begin{align*}
X_{M}<Y_{M},Z>
&=<L_{X_{M}}Y_{M},Z>+<Y_{M},L_{X_{M}}Z>\\
&=<[X_{M},Y_{M}],Z>+<Y_{M},[X_{M},Z]>
\end{align*}
\begin{align*}
Y_{M}<X_{M},Z>
&=<L_{Y_{M}}X_{M},Z>+<X_{M},L_{Y_{M}}Z>\\
&=<[Y_{M},X_{M}],Z>+<X_{M},[Y_{M},Z]>
\end{align*}
then we get
$$(L_{X_{M}}Y^{'}+L_{Y_{M}}X^{'})(Z)=<[X_{M},Y_{M}],Z>+<[Y_{M},X_{M}],Z>=0.$$
If $[X_{M},Y_{M}]=0$, we have
$$(L_{X_{M}}Y^{'})(Z)=<[X_{M},Y_{M}],Z>+<Y_{M},[X_{M},Z]>-<[X_{M},Z],Y_{M}>=0.$$
$$(L_{Y_{M}}X^{'})(Z)=<[Y_{M},X_{M}],Z>+<X_{M},[Y_{M},Z]>-<[Y_{M},Z],X_{M}>=0.$$
\end{proof}

\begin{lemma}
If $X,Y\in\mathfrak{g}$, let $X_{M} ,Y_{M}$ be the corresponding
smooth vector field on $M$, $X^{'}, Y^{'}$ be the 1-form on $M$
which is dual to $X_{M} ,Y_{M}$ by the metric $g^{TM}$, then
$$d_{X+\sqrt{-1}Y}(X^{'}+\sqrt{-1}Y^{'})$$ is the
$d_{X+\sqrt{-1}Y}$-closed form.
\end{lemma}
\begin{proof}
\begin{align*}
d_{X+\sqrt{-1}Y}^{2}(X^{'}+\sqrt{-1}Y^{'})
&=d_{X+\sqrt{-1}Y}(d(X^{'}+\sqrt{-1}Y^{'})+i_{X_{M}+\sqrt{-1}Y_{M}}(X^{'}+\sqrt{-1}Y^{'}))\\
&=di_{X_{M}+\sqrt{-1}Y_{M}}(X^{'}+\sqrt{-1}Y^{'})+i_{X_{M}+\sqrt{-1}Y_{M}}d(X^{'}+\sqrt{-1}Y^{'})\\
&=L_{X_{M}}X^{'}-L_{Y_{M}}Y^{'}+\sqrt{-1}(L_{X_{M}}Y^{'}+L_{Y_{M}}X^{'})\\
&=0
\end{align*}
So $d_{X+\sqrt{-1}Y}(X^{'}+\sqrt{-1}Y^{'})$ is the
$d_{X+\sqrt{-1}Y}$-closed form.
\end{proof}

\begin{lemma}
For any $\eta\in H^{*}_{X+\sqrt{-1}Y}(M)$ and $s\geq 0$, we have
$$\int_{M}\eta=\int_{M}\exp\{-s(d_{X+\sqrt{-1}Y}(X^{'}+\sqrt{-1}Y^{'}))\}\eta$$
\end{lemma}
\begin{proof}
Because $$\frac{\partial}{\partial
s}\int_{M}\exp\{-s(d_{X+\sqrt{-1}Y}(X^{'}+\sqrt{-1}Y^{'}))\}\eta$$
$$=-\int_{M}(d_{X+\sqrt{-1}Y}(X^{'}+\sqrt{-1}Y^{'}))\exp\{-s(d_{X+\sqrt{-1}Y}(X^{'}+\sqrt{-1}Y^{'}))\}\eta$$
and by assumption we have
$$d_{X+\sqrt{-1}Y}\eta=0$$
$$d_{X+\sqrt{-1}Y}\exp\{-s(d_{X+\sqrt{-1}Y}(X^{'}+\sqrt{-1}Y^{'}))\}=0$$
So we get
$$(d_{X+\sqrt{-1}Y}(X^{'}+\sqrt{-1}Y^{'}))\exp\{-s(d_{X+\sqrt{-1}Y}(X^{'}+\sqrt{-1}Y^{'}))\}\eta$$
$$=d_{X+\sqrt{-1}Y}[(X^{'}+\sqrt{-1}Y^{'})\exp\{-s(d_{X+\sqrt{-1}Y}(X^{'}+\sqrt{-1}Y^{'}))\}\eta]$$
and by Stokes formula we have
$$\frac{\partial}{\partial
s}\int_{M}\exp\{-s(d_{X+\sqrt{-1}Y}(X^{'}+\sqrt{-1}Y^{'}))\}\eta=0$$
Then we get
$$\int_{M}\eta=\int_{M}\exp\{-s(d_{X+\sqrt{-1}Y}(X^{'}+\sqrt{-1}Y^{'}))\}\eta$$
\end{proof}

We have
$$d_{X+\sqrt{-1}Y}(X^{'}+\sqrt{-1}Y^{'})
=d(X^{'}+\sqrt{-1}Y^{'})+\langle X_{M}+\sqrt{-1}Y_{M},
X_{M}+\sqrt{-1}Y_{M}\rangle$$ and $$\langle X_{M}+\sqrt{-1}Y_{M},
X_{M}+\sqrt{-1}Y_{M}\rangle=|X_{M}|^{2}-|Y_{M}|^{2}+2\sqrt{-1}\langle
X_{M}, Y_{M}\rangle$$

Set
$$M_{0}=\{x\in M \mid \langle X_{M}(x)+\sqrt{-1}Y_{M}(x),
X_{M}(x)+\sqrt{-1}Y_{M}(x)\rangle=0\}.$$
For simplicity, we assume that $M_{0}$ is the
connected submanifold of $M$, $\dim M_{0}<\dim M$, and $\mathcal{N}$ is the normal bundle
of $M_{0}$ about $M$. The set $M_{0}$ is first discussed by
H.Jacobowitz (see [4]).

\begin{lemma}
For any $\eta\in H^{*}_{X+\sqrt{-1}Y}(M)$, the top form of $\eta$ is
exact outside $M_{0}$.
\end{lemma}
\begin{proof}
By lemma 3., we have on $M\backslash M_{0}$
$$\eta=d_{X+\sqrt{-1}Y}\left(\frac{(X^{'}+\sqrt{-1}Y^{'})\wedge\eta}{d_{X+\sqrt{-1}Y}(X^{'}+\sqrt{-1}Y^{'})}\right).$$
If $\eta_{[n]}$ is the top form of $\eta$, then we get
$$\eta_{[n]}=d\left(\frac{(X^{'}+\sqrt{-1}Y^{'})\wedge\eta}{d_{X+\sqrt{-1}Y}(X^{'}+\sqrt{-1}Y^{'})}\right)_{[n-1]}$$
on $M\backslash M_{0}$.
\end{proof}

\begin{corollery}
For any $\eta\in H^{*}_{X+\sqrt{-1}Y}(M)$, if $M_{0}=\emptyset$, then $\int_{M}\eta=0$.
\end{corollery}
\begin{proof}
By the Stokes formula and Lemma 5., we get the result.
\end{proof}

\section{Localization formula on $d_{X+\sqrt{-1}Y}$-closed form}
Set $E$ is a G-equivariant vector bundle, if $\nabla^{E}$ is a
connection on $E$ which commutes with the action of $G$ on
$\Omega(M,E)$, we see that $$[\nabla^{E}, L^{E}_{X}]=0$$ for all
$X\in\mathfrak{g}$. Then we can get a moment map by
$$\mu^{E}(X)=L^{E}_{X}-[\nabla^{E},i_{X}]=L^{E}_{X}-\nabla^{E}_{X}$$
We known that if $y$ be the tautological section of the bundle
$\pi^{*}E$ over E, then the vertical component of $X_{E}$ may be
identified with $-\mu^{E}(X)y$(see [1] proposition 7.6). For the normal bundle $\mathcal{N}$ of $M_0$, the
vector field $X^{\mathcal{N}}$ and $Y^{\mathcal{N}}$ are vertical
and are given at the point $(x,y)\in M_{0}\times\mathcal{N}_{x}$ by
the vectors $-\mu^{\mathcal{N}}(X)y,
-\mu^{\mathcal{N}}(Y)y\in\mathcal{N}_{x}$.\par

If $E$ is the tangent bundle $TM$ and $\nabla^{TM}$ is Levi-Civita
connection, then we have
$$\mu^{TM}(X)Y=L_{X}Y-\nabla^{TM}_{X}Y=-\nabla^{TM}_{Y}X$$
We known that for any Killing vector field $X$, $\mu^{TM}(X)$ as
linear endomorphisms of $TM$ is skew-symmetric, $-\mu^{TM}(X)$
annihilates the tangent bundle $TM_{0}$ and induces a skew-symmetric
automorphism of the normal bundle $\mathcal{N}$(see [5] chapter II,
proposition 2.2 and theorem 5.3). The restriction of $\mu^{TM}(X)$
to $\mathcal{N}$ coincides with the moment endomorphism
$\mu^{\mathcal{N}}(X)$. 

Now we construct a one-form $\alpha$ on $\mathcal{N}$:
$$Z\in \Gamma(T\mathcal{N})\rightarrow\alpha(Z)=\langle-\mu^{\mathcal{N}}(X)y,\nabla^{\mathcal{N}}_{Z}y\rangle+\sqrt{-1}\langle-\mu^{\mathcal{N}}(Y)y,\nabla^{\mathcal{N}}_{Z}y\rangle$$

Let $Z_{1},Z_{2}\in\Gamma(T\mathcal{N})$, we known
$d\alpha(Z_{1},Z_{2})=Z_{1}\alpha(Z_{2})-Z_{2}\alpha(Z_{1})-\alpha([Z_{1},Z_{2}])$,
so
\begin{align*}
d\alpha(Z_{1},Z_{2})
&=\langle-\nabla^{\mathcal{N}}_{Z_{1}}\mu^{\mathcal{N}}(X)y,\nabla^{\mathcal{N}}_{Z_{2}}y\rangle-\langle-\nabla^{\mathcal{N}}_{Z_{2}}\mu^{\mathcal{N}}(X)y,\nabla^{\mathcal{N}}_{Z_{1}}y\rangle\\
&+\sqrt{-1}\langle-\nabla^{\mathcal{N}}_{Z_{1}}\mu^{\mathcal{N}}(Y)y,\nabla^{\mathcal{N}}_{Z_{2}}y\rangle-\sqrt{-1}\langle-\nabla^{\mathcal{N}}_{Z_{2}}\mu^{\mathcal{N}}(Y)y,\nabla^{\mathcal{N}}_{Z_{1}}y\rangle\\
&+\langle-\mu^{\mathcal{N}}(X)y,R^{\mathcal{N}}(Z_{1},Z_{2})y\rangle+\sqrt{-1}\langle-\mu^{\mathcal{N}}(Y)y,R^{\mathcal{N}}(Z_{1},Z_{2})y\rangle\\
\end{align*}
Recall that $\nabla^{\mathcal{N}}$ is invariant under $L_{X}$ for
all $X\in\mathfrak{g}$, so that
$[\nabla^{\mathcal{N}},\mu^{\mathcal{N}}(X)]=0$,
$[\nabla^{\mathcal{N}},\mu^{\mathcal{N}}(Y)]=0$. And by $X,Y$ are
Killing vector field, we have $d\alpha$ equals
$$2\langle-(\mu^{\mathcal{N}}(X)+\sqrt{-1}\mu^{\mathcal{N}}(Y))
\cdot,\cdot\rangle+\langle-\mu^{\mathcal{N}}(X)y-\sqrt{-1}\mu^{\mathcal{N}}(Y)y,R^{\mathcal{N}}y\rangle$$
And by
$|X_{\mathcal{N}}|^{2}=\langle\mu^{\mathcal{N}}(X)y,\mu^{\mathcal{N}}(X)y\rangle$,
$|Y_{\mathcal{N}}|^{2}=\langle\mu^{\mathcal{N}}(Y)y,\mu^{\mathcal{N}}(Y)y\rangle$.
So We can get
\begin{align*}
d_{X_{\mathcal{N}}+\sqrt{-1}Y_{\mathcal{N}}}(X^{'}_{\mathcal{N}}+\sqrt{-1}Y^{'}_{\mathcal{N}})
&=-2\langle(\mu^{\mathcal{N}}(X)+\sqrt{-1}\mu^{\mathcal{N}}(Y))\cdot,\cdot\rangle\\
&+\langle-\mu^{\mathcal{N}}(X)y-\sqrt{-1}\mu^{\mathcal{N}}(Y)y,-\mu^{\mathcal{N}}(X)y-\sqrt{-1}\mu^{\mathcal{N}}(Y)y+R^{\mathcal{N}}y\rangle
\end{align*}

\begin{theorem}
Let $M$ be a smooth closed oriented manifold, $G$ be a compact Lie group
acting smoothly on $M$. For any $\eta\in H^{*}_{X+\sqrt{-1}Y}(M)$, $[X_{M},Y_{M}]=0$, the following identity
hold:
$$\int_{M}\eta=\int_{M_{0}} \frac{\eta}{\rm{Pf}[\frac{-\mu^{\mathcal{N}}(X)-\sqrt{-1}\mu^{\mathcal{N}}(Y)+R^{\mathcal{N}}}{2\pi}]}$$
\end{theorem}
\begin{proof}
Set $s=\frac{1}{2t}$, so by Lemma 4. we get
$$\int_{M}\eta=\int_{M}\exp\{-\frac{1}{2t}(d_{X+\sqrt{-1}Y}(X^{'}+\sqrt{-1}Y^{'}))\}\eta$$
Let $V$ is a neighborhood of $M_{0}$ in $\mathcal{N}$. We identify a
tubular neighborhood of $M_{0}$ in $M$ with $V$. Set $V^{'}\subset
V$. When $t\rightarrow 0$, because $\langle
X_{M}(x)+\sqrt{-1}Y_{M}(x), X_{M}(x)+\sqrt{-1}Y_{M}(x)\rangle\neq0$
out of $M_{0}$, so we have
$$\int_{M}\exp\{-\frac{1}{2t}(d_{X+\sqrt{-1}Y}(X^{'}+\sqrt{-1}Y^{'}))\}\eta\sim\int_{V^{'}}\exp\{-\frac{1}{2t}(d_{X+\sqrt{-1}Y}(X^{'}+\sqrt{-1}Y^{'}))\}\eta.$$
Because
$$\int_{V^{'}}\exp\{-\frac{1}{2t}(d_{X+\sqrt{-1}Y}(X^{'}+\sqrt{-1}Y^{'}))\}\eta=\int_{V^{'}}\exp\{-\frac{1}{2t}(d_{X_{\mathcal{N}}+\sqrt{-1}Y_{\mathcal{N}}}(X_{\mathcal{N}}^{'}+\sqrt{-1}Y_{\mathcal{N}}^{'}))\}\eta$$
then
$$\int_{V^{'}}\exp\{-\frac{1}{2t}(d_{X+\sqrt{-1}Y}(X^{'}+\sqrt{-1}Y^{'}))\}\eta=$$
$$\int_{V^{'}}\exp\{\frac{1}{t}\langle(\mu^{\mathcal{N}}(X)+\sqrt{-1}\mu^{\mathcal{N}}(Y))\cdot,\cdot\rangle+\frac{1}{2t}\langle\mu^{\mathcal{N}}(X)y+\sqrt{-1}\mu^{\mathcal{N}}(Y)y,R^{\mathcal{N}}y\rangle\}\eta$$
$$+\int_{V^{'}}\exp\{-\frac{1}{2t}\langle-\mu^{\mathcal{N}}(X)y-\sqrt{-1}\mu^{\mathcal{N}}(Y)y, -\mu^{\mathcal{N}}(X)y-\sqrt{-1}\mu^{\mathcal{N}}(Y)y\rangle\}\eta$$
By making the change of variables $y=\sqrt{t}y$, we find that the
above formula is equal to
$$t^{n}\int_{V^{'}}\exp\{\frac{1}{t}\langle(\mu^{\mathcal{N}}(X)+\sqrt{-1}\mu^{\mathcal{N}}(Y))\cdot,\cdot\rangle+\frac{1}{2}\langle\mu^{\mathcal{N}}(X)y+\sqrt{-1}\mu^{\mathcal{N}}(Y)y,R^{\mathcal{N}}y\rangle\}\eta$$
$$+\int_{V^{'}}\exp\{-\frac{1}{2}\langle-\mu^{\mathcal{N}}(X)y-\sqrt{-1}\mu^{\mathcal{N}}(Y)y, -\mu^{\mathcal{N}}(X)y-\sqrt{-1}\mu^{\mathcal{N}}(Y)y\rangle\}\eta_{\sqrt{t}y}$$
we known that
$$\frac{(\frac{\langle(\mu^{\mathcal{N}}(X)+\sqrt{-1}\mu^{\mathcal{N}}(Y))\cdot,\cdot\rangle}{t})^{n}}{n!}=(\rm{Pf}(\mu^{\mathcal{N}}(X)+\sqrt{-1}\mu^{\mathcal{N}}(Y)))dy$$
here dy is the volume form of the submanifold $M_{0}$, let $n$ be the
dimension of $M_{0}$, then we get
$$=\int_{V^{'}}\exp\{\frac{1}{2}\langle\mu^{\mathcal{N}}(X)y+\sqrt{-1}\mu^{\mathcal{N}}(Y)y,R^{\mathcal{N}}y\rangle\}\eta[\det(\mu^{\mathcal{N}}(X)+\sqrt{-1}\mu^{\mathcal{N}}(Y))]^{\frac{1}{2}}dy_{1}\wedge...\wedge dy_{n}$$
$$+\int_{V^{'}}\exp\{-\frac{1}{2}\langle-\mu^{\mathcal{N}}(X)y-\sqrt{-1}\mu^{\mathcal{N}}(Y)y, -\mu^{\mathcal{N}}(X)y-\sqrt{-1}\mu^{\mathcal{N}}(Y)y\rangle\}\eta$$
Because by $[X_{M},Y_{M}]=0$ we have $[\mu^{TM}(X),\mu^{TM}(Y)]=0$.
And by $-\mu^{\mathcal{N}}(X)-\sqrt{-1}\mu^{\mathcal{N}}(Y)$,
$R^{\mathcal{N}}$ are skew-symmetric, so we get
$$=\int_{V^{'}}\exp\{-\frac{1}{2}\langle-\mu^{\mathcal{N}}(X)y-\sqrt{-1}\mu^{\mathcal{N}}(Y)y,-\mu^{\mathcal{N}}(X)y-\sqrt{-1}\mu^{\mathcal{N}}(Y)y+R^{\mathcal{N}}y\rangle\}dy_{1}\wedge...\wedge dy_{n}$$
$$\cdot[\det(\mu^{\mathcal{N}}(X)+\sqrt{-1}\mu^{\mathcal{N}}(Y))]^{\frac{1}{2}}\eta$$
$$=\int_{M_{0}}(2\pi)^{n}[\det(\mu^{\mathcal{N}}(X)+\sqrt{-1}\mu^{\mathcal{N}}(Y))]^{-\frac{1}{2}}[\det(-\mu^{\mathcal{N}}(X)-\sqrt{-1}\mu^{\mathcal{N}}(Y)+R^{\mathcal{N}})]^{-\frac{1}{2}}$$
$$\cdot[\det(\mu^{\mathcal{N}}(X)+\sqrt{-1}\mu^{\mathcal{N}}(Y))]^{\frac{1}{2}}\eta$$
$$=\int_{M_{0}}(2\pi)^{n}[\det(-\mu^{\mathcal{N}}(X)-\sqrt{-1}\mu^{\mathcal{N}}(Y)+R^{\mathcal{N}})]^{-\frac{1}{2}}\eta$$
$$=\int_{M_{0}} \frac{\eta}{\rm{Pf}[\frac{-\mu^{\mathcal{N}}(X)-\sqrt{-1}\mu^{\mathcal{N}}(Y)+R^{\mathcal{N}}}{2\pi}]}$$
\end{proof}
By theorem 1.,we can get the localization formulas of Berline and Vergne(see [1] or [2]).
\begin{corollery}[N.Berline and M.Vergne]
Let $M$ be a smooth closed oriented manifold, $G$ be a compact Lie
group acting smoothly on $M$. For any $\eta\in H^{*}_{X}(M)$, the following identity
hold:
$$\int_{M}\eta=\int_{M_{0}} \frac{\eta}{\rm{Pf}[\frac{-\mu^{\mathcal{N}}(X)+R^{\mathcal{N}}}{2\pi}]}$$
\end{corollery}
\begin{proof}
By theorem 1., we set $Y=0$, then we get the result.
\end{proof}

\section{Localization formulas for characteristic numbers}
Let $M$ be an even dimensional compact oriented manifold without boundary, $G$ be a compact Lie group
acting smoothly on $M$ and $\mathfrak{g}$
be its Lie algebra. Let $g^{TM}$ be a $G$-invariant Riemannian metric on $TM$, $\nabla^{TM}$ is the Levi-Civita connection associated to $g^{TM}$. Here $\nabla^{TM}$ is a $G$-invariant connection, we see that $[\nabla^{TM},L_{X_{M}}]=0$ for all $X\in\mathfrak{g}$.\par

The equivariant connection $\widetilde{\nabla}^{TM}$ is the operator on $\Omega^{*}(M,TM)$ corresponding to a $G$-invariant connection $\nabla^{TM}$ is defined by the formula
$$\widetilde{\nabla}^{TM}=\nabla^{TM}+i_{X_{M}+\sqrt{-1}Y_{M}}$$
here $X_{M} ,Y_{M}$ be the smooth vector field on $M$ corresponded to $X,Y\in\mathfrak{g}$.
\begin{lemma}
The operator $\widetilde{\nabla}^{TM}$ preserves the space $\Omega^{*}_{X_{M}+\sqrt{-1}Y_{M}}(M,TM)$ which is the space of smooth $(X_{M}+\sqrt{-1}Y_{M})$-invariant forms with values in $TM$.
\end{lemma}
\begin{proof}
Let $\omega\in\Omega^{*}_{X_{M}+\sqrt{-1}Y_{M}}(M)$, then we have
\begin{align*}
(L_{X_{M}}+\sqrt{-1}L_{Y_{M}})\widetilde{\nabla}^{TM}\omega
&=(L_{X_{M}}+\sqrt{-1}L_{Y_{M}})(\nabla^{TM}+i_{X_{M}+\sqrt{-1}Y_{M}})\omega\\
&=(\nabla^{TM}+i_{X_{M}+\sqrt{-1}Y_{M}})(L_{X_{M}}+\sqrt{-1}L_{Y_{M}})\omega\\
&=0
\end{align*}
So we get $\widetilde{\nabla}^{TM}\omega\in\Omega^{*}_{X_{M}+\sqrt{-1}Y_{M}}(M,TM)$.
\end{proof}
We will also denote the restriction of $\widetilde{\nabla}^{TM}$ to $\Omega^{*}_{X_{M}+\sqrt{-1}Y_{M}}(M,TM)$ by $\widetilde{\nabla}^{TM}$.

The equivariant curvature $\widetilde{R}^{TM}$ of the equivariant connection $\widetilde{\nabla}^{TM}$ is defined by the formula(see [1])
$$\widetilde{R}^{TM}=(\widetilde{\nabla}^{TM})^{2}-L_{X_{M}}-\sqrt{-1}L_{Y_{M}}$$
It is the element of $\Omega^{*}_{X_{M}+\sqrt{-1}Y_{M}}(M,End(TM))$. We see that
\begin{align*}
\widetilde{R}^{TM}
&=(\nabla^{TM}+i_{X_{M}+\sqrt{-1}Y_{M}})^{2}-L_{X_{M}}-\sqrt{-1}L_{Y_{M}}\\
&=R^{TM}+[\nabla^{TM},i_{X_{M}+\sqrt{-1}Y_{M}}]-L_{X_{M}}-\sqrt{-1}L_{Y_{M}}\\
&=R^{TM}-\mu^{TM}(X)-\sqrt{-1}\mu^{TM}(Y)
\end{align*}

\begin{lemma}
The equivariant curvature $\widetilde{R}^{TM}$ satisfies the equvariant Bianchi formula $$\widetilde{\nabla}^{TM}\widetilde{R}^{TM}=0$$
\end{lemma}
\begin{proof}
Because
\begin{align*}
[\widetilde{\nabla}^{TM},\widetilde{R}^{TM}]
&=[\widetilde{\nabla}^{TM},(\widetilde{\nabla}^{TM})^{2}-L_{X_{M}}-\sqrt{-1}L_{Y_{M}}]\\
&=[\widetilde{\nabla}^{TM},(\widetilde{\nabla}^{TM})^{2}]+[\nabla^{TM}+i_{X_{M}+\sqrt{-1}Y_{M}},-L_{X_{M}}-\sqrt{-1}L_{Y_{M}}]\\
&=0
\end{align*}
\end{proof}

Now we to construct the equivariant characteristic forms by $\widetilde{R}^{TM}$. If $f(x)$ is a polynomial in the indeterminate $x$, then $f(\widetilde{R}^{TM})$ is an element of $\Omega^{*}_{X_{M}+\sqrt{-1}Y_{M}}(M,End(TM))$. We use the trace map
$${\rm Tr}: \Omega^{*}_{X_{M}+\sqrt{-1}Y_{M}}(M,End(TM))\rightarrow\Omega^{*}_{X_{M}+\sqrt{-1}Y_{M}}(M)$$
to obtain an element of $\Omega^{*}_{X_{M}+\sqrt{-1}Y_{M}}(M)$, which we call an equivariant characteristic form.
\begin{lemma}
The equivariant differential form ${\rm Tr}(f(\widetilde{R}^{TM}))$ is $d_{X_{M}+\sqrt{-1}Y_{M}}$-closed, and its equivariant cohomology class is independent of the choice of the G-invariant connection $\nabla^{TM}$.
\end{lemma}
\begin{proof}
If $\alpha\in\Omega^{*}_{X_{M}+\sqrt{-1}Y_{M}}(M,End(TM))$, because in local $\nabla^{TM}=d+\omega$, we have
\begin{align*}
d_{X_{M}+\sqrt{-1}Y_{M}}{\rm Tr}(\alpha)
&={\rm Tr}(d_{X_{M}+\sqrt{-1}Y_{M}}\alpha)\\
&={\rm Tr}([d_{X_{M}+\sqrt{-1}Y_{M}},\alpha])+{\rm Tr}([\omega,\alpha])\\
&={\rm Tr}([\widetilde{\nabla}^{TM},\alpha])
\end{align*}
then by the equivariant Bianchi identity $\widetilde{\nabla}^{TM}\widetilde{R}^{TM}=0$, we get
$$d_{X_{M}+\sqrt{-1}Y_{M}}{\rm Tr}(f(\widetilde{R}^{TM}))=0.$$

Let $\nabla^{TM}_{t}$ is a one-parameter family of G-invariant connections with equivariant curvature $\widetilde{R}^{TM}_{t}$. We have
\begin{align*}
\frac{d}{dt}{\rm Tr}(f(\widetilde{R}^{TM}_{t}))
&={\rm Tr}(\frac{d\widetilde{R}^{TM}_{t}}{dt}f^{'}(\widetilde{R}^{TM}_{t}))\\
&={\rm Tr}(\frac{d(\widetilde{\nabla}^{TM}_{t})^{2}}{dt}f^{'}(\widetilde{R}^{TM}_{t}))\\
&={\rm Tr}([\widetilde{\nabla}^{TM}_{t},\frac{d\widetilde{\nabla}^{TM}_{t}}{dt}]f^{'}(\widetilde{R}^{TM}_{t}))\\
&={\rm Tr}([\widetilde{\nabla}^{TM}_{t},\frac{d\widetilde{\nabla}^{TM}_{t}}{dt}f^{'}(\widetilde{R}^{TM}_{t})])\\
&=d_{X_{M}+\sqrt{-1}Y_{M}}{\rm Tr}(\frac{d\widetilde{\nabla}^{TM}_{t}}{dt}f^{'}(\widetilde{R}^{TM}_{t}))
\end{align*}
from which we get
$${\rm Tr}(f(\widetilde{R}^{TM}_{1}))-{\rm Tr}(f(\widetilde{R}^{TM}_{0}))=d_{X_{M}+\sqrt{-1}Y_{M}}\int^{1}_{0}{\rm Tr}(\frac{d\widetilde{\nabla}^{TM}_{t}}{dt}f^{'}(\widetilde{R}^{TM}_{t}))dt$$
so we get the result.
\end{proof}

As an application of Theorem 1., we can get the following localization formulas for characteristic numbers
\begin{theorem}
Let $M$ be an $2l$-dim compact oriented manifold without boundary, $G$ be a compact Lie group
acting smoothly on $M$ and $\mathfrak{g}$ be its Lie algebra. Let $X,Y\in\mathfrak{g}$, and $X_{M} ,Y_{M}$ be the corresponding
smooth vector field on $M$. $M_{0}$ is the submanifold descriped in section 2. If $f(x)$ is a polynomial, then we have
$$\int_{M}{\rm Tr}(f(\widetilde{R}^{TM}))=\int_{M_{0}}\frac{{\rm Tr}(f(\widetilde{R}^{TM}))}{\rm{Pf}[\frac{-\mu^{\mathcal{N}}(X)-\sqrt{-1}\mu^{\mathcal{N}}(Y)+R^{\mathcal{N}}}{2\pi}]}$$
\end{theorem}
\begin{proof}
By Lemma 8., we have ${\rm Tr}(f(\widetilde{R}^{TM}))$ is $d_{X_{M}+\sqrt{-1}Y_{M}}$-closed. And by Theorem 1., we get the result.
\end{proof}
We can use this formula to compute these characteristic numbers of $M$, especially we can use it to Euler characteristic of $M$. Here we didn't to give the details.

\end{document}